\documentclass[12pt]{article}
\usepackage{enumerate}
\usepackage[hyperindex,breaklinks,colorlinks,citecolor=blue,pagebackref]{hyperref}
\usepackage[latin1]{inputenc}
\usepackage{amsmath, amsthm, amsfonts, amssymb}
\usepackage{mathrsfs}
\usepackage{latexsym}
\usepackage{fullpage}
\usepackage[all]{xy}
\usepackage{color}
\usepackage{stmaryrd}

\newtheorem{theorem}{Theorem}[section]

\newtheorem{definition}[theorem]{Definition}

\theoremstyle{remark}

\newtheorem{rmks}{Remarks}[section]
\newtheorem*{Examples}{Examples}

\newtheorem*{claim}{Claim}

\newcommand{\C}{\mathcal{C}}

\newcommand{\G}{\mathcal{G}}

\newcommand{\N}{\mathbb{N}}

\renewcommand{\P}{\mathcal{P}}
\newcommand{\Q}{\mathcal{Q}}

\newcommand{\U}{\mathcal{U}}
\newcommand{\V}{\mathcal{V}}
\newcommand{\W}{\mathcal{W}}

\newcommand{\prefx}{\preceq}
\newcommand{\exts}{\succeq}

\newcommand{\Nat}{\mathrm{Nat}}

\definecolor{purple}{rgb}{.9,0.2,.9}

\newcommand{\cs}{2^\omega}
\newcommand{\uh}{{\upharpoonright}}
\renewcommand{\phi}{\varphi}
\newcommand{\str}{2^{<\omega}}

\newcommand{\leqx}{\preceq}

\newcommand{\tuple}[1]{\langle #1 \rangle}

\usepackage{xcolor}	
	\usepackage{soul}

	\definecolor{lightblue}{rgb}{.60,.60,1}

\title{On the interplay between effective notions of randomness and genericity}
\author{Laurent Bienvenu, Christopher P. Porter}

\begin{document}

\maketitle

\begin{abstract}
In this paper, we study the power and limitations of computing effectively generic sequences using effectively random oracles.  Previously, it was known that every 2-random sequence computes a 1-generic sequence (as shown by Kautz) and every 2-random sequence forms a minimal pair in the Turing degrees with every 2-generic sequence (as shown by Nies, Stephan, and Terwijn).  We strengthen these results by showing that every Demuth random sequence computes a 1-generic sequence (which answers an open question posed by Barmpalias, Day, and Lewis) and that every Demuth random sequence forms a minimal pair with every pb-generic sequence (where pb-genericity is an effective notion of genericity that is strictly between 1-genericity and 2-genericity).  Moreover, we prove that for every comeager $\G\subseteq \cs$, there is some weakly 2-random sequence $X$ that computes some $Y\in\G$, a result that allows us to provide a fairly complete classification as to how various notions of effective randomness interact in the Turing degrees with various notions of effective genericity.

\end{abstract}

\section{Introduction}

Randomness and genericity play an important role in computability theory in that they both, in their own way, define what it means for an infinite binary sequence~$X$ to be typical among all infinite binary sequences. For any reasonable way of defining randomness and genericity, these two notions are orthogonal, i.e., a random sequence cannot be generic and a generic sequence cannot be random. Moreover, for sufficient levels of randomness and genericity, this orthogonality goes even further: Nies, Stephan and Terwijn~\cite{NiesST2005} showed that a 2-random sequence and a 2-generic sequence always form a minimal pair in the Turing degrees.

At lower levels of randomness and genericity, however, the situation is more nuanced. For example, by the Ku\v cera-G\'acs theorem, any sequence is computed by some $1$-random sequence, thus, in particular, any $n$-generic sequence (for any~$n$) is computed by some $1$-random sequence. Another striking result, due to Kautz~\cite{Kautz1991} (building on the work of Kurtz~\cite{Kurtz1981}), is that \emph{every} 2-random sequence \emph{must} compute some $1$-generic sequence. It therefore makes sense to ask how random sequences and generic sequences interact for levels of randomness between $1$-randomness and $2$-randomness, and for levels of genericity between $1$-genericity and $2$-genericity. This is precisely the purpose of this paper. As we will see, the notion of randomness that has the most interesting interactions with genericity turns out to be Demuth randomness. In particular, we answer positively a question of Barmpalias, Day, and Lewis~\cite{BarmpaliasDL2013}, who asked whether every Demuth random sequence computes a $1$-generic sequence.

\section{Notation and background}

While we assume the reader is familiar with computability theory, let us briefly recall some basic definitions. We work in the Cantor space $\cs$, that is, the space of infinite binary sequences endowed with the product topology, i.e., the topology generated by the cylinders~$[\sigma]$, where $\sigma$ is a finite binary sequence (also referred to as \emph{string}) and $[\sigma]$ is the subset of $\cs$ consisting of the $X$'s that have $\sigma$ as a prefix. We denote by $\str$ the set of all binary strings, $\Lambda$ the empty string. For a string $\sigma$, $|\sigma|$ is the length of $\sigma$ and if $n \leq |\sigma|$, $\sigma \uh n$ is the prefix of $\sigma$ of length~$n$. We also use $X \uh n$ when $X \in \cs$ to denote the prefix of $X$ of length~$n$. The prefix relation is denoted by $\leqx$. An open subset of $\cs$ is a set of type $\bigcup_{\sigma \in S} [\sigma]$ for some countable set~$S$ of strings; when $S$ is computably enumerable, we say that the open set is $\emph{effectively open}$. 

Let us briefly recall the definitions of genericity and of weak genericity. For $n \geq 1$, we say that $X$ is \emph{weakly $n$-generic} if $X$ belongs to every dense open set $\U$ that is effectively open relative to $\emptyset^{(n-1)}$. We say that \emph{$n$-generic} if for every open set $\U$ that is effectively open relative to $\emptyset^{(n-1)}$, $X$ belongs to $\U \cup \overline{\U}^c$ (where $\overline{\U}^c$ is the complement of the closure of $\U$). Equivalently, $X$ is weakly $n$-generic if for every dense set $S$ of strings\footnote{Here ``dense" should be understood relative to the prefix order: a set $S$ of strings is dense if every string $\sigma\in\str$ has an extension in~$S$.} that is c.e.\ relative to $\emptyset^{(n-1)}$, some prefix of $X$ is in~$S$, and $X$ is $n$-generic if for set $S$ of strings that is c.e.\ relative to $\emptyset^{(n-1)}$, either some prefix of $X$ is in~$S$ or  some prefix of~$X$ has no extension in~$S$. For every~$n$, the following relations hold:

\begin{center}
weak $(n+1)$-genericity $\Rightarrow$ $n$-genericity $\Rightarrow$ weak $n$-genericity
\end{center}

The Lebesgue, or uniform, measure $\mu$ on $\cs$ is the measure corresponding to the random process where each bit has value $0$ with probability $1/2$ independently of all other bits. Equivalently, $\mu$ is the unique measure on~$\cs$ such that $\mu([\sigma])= 2^{-|\sigma|}$ for all~$\sigma \in \str$. For any measurable subset $\mathcal{A}$ of $\cs$ and $\sigma \in \str$, we write $\mu(\mathcal{A} \mid \sigma)$ for $\mu(\mathcal{A} \cap [\sigma])/\mu([\sigma])$. When we talk about `randomness' or `random' objects, we implicitly mean `with respect to Lebesgue measure'. 

A \emph{Martin-L\"of test} is a sequence $(\U_n)_{n \in \omega}$ of uniformly effectively open sets such that $\mu(\U_n) \leq 2^{-n}$ for all~$n$. An $X \in \cs$ is \emph{Martin-L\"of random} if for every Martin-L\"of test, $X \notin \bigcap_n \U_n$.
% for Martin-L\"of randomness, it is equivalent to replace this last condition by `$X$ only belongs to finitely many $\U_n$'s'. 
The other notions of randomness we will encounter in this paper will be recalled as we proceed.

\section{Demuth randomness vs. effective genericity}

In this section, we shall see how Demuth randomness interacts with various notions of effective genericity. Our main result is that every Demuth random sequence computes a $1$-generic sequence. The proof has two components. First, we shall review the classical proof that a sufficiently random sequence computes a $1$-generic sequence, and show that the failure set (the set of $X$'s that fail to compute a $1$-generic) can be covered by a specific type of randomness test. We will then show that tests of that type in fact characterize Demuth randomness and obtain the desired result.  

\subsection{Fireworks arguments}

Kautz's proof that almost every~$X\in\cs$ computes a $1$-generic sequence is framed in a way that is difficult to precisely analyze in terms of algorithmic randomness (i.e., to determine how random~$X$ needs to be for the argument to work). A more intuitive proof can be given using a \emph{fireworks argument} (which takes its name from the presentation by Rumyantsev and Shen~\cite{RumyantsevS2014} with an analogy about purchasing fireworks from a purportedly corrupt fireworks salesman), an approach that is more suitable for our purposes. The mechanics of fireworks arguments are already thoroughly explained in~\cite{RumyantsevS2014} and \cite{BienvenuP2017}, but we shall review them for the sake of completeness. 

For now, we will set aside notions of algorithmic randomness and state the result we want to prove as follows: ``For every $k$, we can, uniformly in~$k$, design a probabilistic algorithm that produces a $1$-generic sequence with probability $\geq 1-2^{-k}$."

Let us thus fix~$k\in\omega$. Let $(W_e)_{e\in\omega}$ be an effective enumeration of all c.e.\ sets of strings.  We must satisfy for each $e\in\omega$ the following requirement:
\begin{itemize}
\item[] $\mathcal{R}_e$:  There is some $\sigma\prec \Phi_k^X$ such that either $\sigma\in W_e$ or for all $\tau\succeq\sigma,\  \tau\notin W_e$.
\end{itemize}

The strategy to satisfy requirement $\mathcal{R}_e$ is as follows. When the strategy receives attention for the first time, the first thing it does is pick an integer $n(e,k)$ at random between $1$ and $N(e,k)$ (all integers in this interval being assigned the same probability), where $N(\cdot,\cdot)$ is a fixed computable function to be specified later. The strategy then makes a \emph{passive guess} that there is no extension of the current prefix $\sigma$ of~$X$ in $W_e$, and moves on to building~$X$ by taking care of strategies for other requirements. If this guess is correct, the requirement $\mathcal{R}_e$ is satisfied and the strategy will have succeeded without having to do anything (this is why we call our guess ``passive", as it requires no action from the strategy). Of course this passive guess may also be incorrect; that is, there may be some extension of $\sigma$ in~$W_e$. If that is the case, this will become apparent at some point (because $W_e$ is c.e.), but by the time it does, a longer prefix $\sigma'$ of~$X$ may have been built that does not have an extension in~$W_e$. The strategy then makes a second passive guess that $\sigma'$ has no extension in~$W_e$, and again moves on to other strategies until this second guess is proven wrong. Again it will make a new passive guess, and so on. 

Now one needs to avoid the undesirable case where the strategy makes infinitely many wrong passive guesses during the construction, for otherwise it might fail to satisfy~$\mathcal{R}_e$. To avoid this situation, we use $n(e,k)$ as a cap on the number of passive guesses that the strategy is allowed to make. Once the strategy has made $n(e,k)$ passive guesses and each time has realized that the guess was wrong, it will then make an \emph{active guess}; that is, it will guess that there \emph{is} an extension~$\tau$ of the current prefix $\sigma$ of~$X$ in $W_e$ and will wait for such an extension to be enumerated into $W_e$.  While waiting, the actions of all the other strategies are put on hold. If indeed an extension $\tau$ of $\sigma$ is enumerated into $W_e$, take~$\tau$ as the new prefix of~$X$, declare $\mathcal{R}_e$ to be satisfied and terminate (allowing the other active strategies to resume). 

There are three possible outcomes for our strategy:
\begin{itemize}
\item[(i)] After some wrong passive guesses, the strategy eventually makes a correct passive guess.
\item[(ii)] After making $n(e,k)$ wrong passive guesses, the strategy makes a correct active guess. 
\item[(iii)] After making $n(e,k)$ wrong passive guesses, the strategy makes a wrong active guess. 
\end{itemize}

As we have seen, the first two outcomes ensure the satisfaction of the requirement~$\mathcal{R}_e$. The third outcome is the bad case: if the strategy makes a wrong active guess, it will wait in vain for an extension of the current prefix of~$X$ to appear in $W_e$, and all the actions of other strategies are stopped during this waiting period, so the algorithm fails to produce an infinite binary sequence~$X$. We claim that the probability that the strategy for $\mathcal{R}_e$ has outcome (iii) is at most $\frac{1}{N(e,k)}$. Indeed, assuming all values $n(e',k)$ for $e' \not= e$ have been chosen, there at most one value $n(e,k)$ can take that would cause the strategy to have outcome (iii). To see why this is the case, we make the trivial observation that when a strategy makes a wrong passive guess, if it had instead made an active guess, this active guess would have been correct, and vice versa. Thus if the strategy ends up in case (iii) after making $m$ wrong passive guesses and one last active guess, then any cap $m'<m$ on the number of passive guesses would have given outcome (ii) instead, and any $m'>m$ would have given outcome (i), as the $m$-th passive guess would have been correct. 

This shows that, conditional to \emph{any fixed choice} of the $n(e',k)$ for $e' \not= e$, the probability for the  $\mathcal{R}_e$-strategy to have outcome (iii) is at most $\frac{1}{N(e,k)}$, since the value of $n(e,k)$ is chosen randomly between $1$ and $N(e,k)$, independently of the $n(e',k)$ for $e' \not= e$. Thus the unconditional probability for the $\mathcal{R}_e$-strategy to have outcome (iii) is at most $\frac{1}{N(e,k)}$. (Technically, we cannot really talk about conditional probability since any fixed choice of $n(e',k)$ for $e' \not= e$ is a probability-$0$ event; what we are actually appealing to is Fubini's theorem, which says that for $f(\cdot,\cdot)$ a measurable function defined on the product of two probability spaces, $\int_x \int_y f(x,y) dx dy = \int_y \int_x f(x,y) dy dx$. In particular, if for every fixed $y$ we have $\int_x f(x,y) dx \leq \varepsilon$, then $\int_{(x,y)} f(x,y) dx dy \leq \varepsilon$; in our case $x$ is $n(e,k)$, $y$ is the sequence of $n(e',k)$ for $e' \not= e$ and  $f$ is the characteristic function of the failure set of the algorithm). 

Over all strategies, the probability of failure of our algorithm is therefore bounded by $\sum_e \frac{1}{N(e,k)}$. So if we take, for example, $N(e,k) = 2^{e+k+1}$, the probability of failure is at most $2^{-k}$, as desired. 

\subsection{How much randomness is needed for fireworks arguments?}

The probabilistic algorithm with parameter~$k$ presented above can be interpreted by a Turing functional $\Psi_k$ which has access to an oracle~$X\in\cs$ chosen at random (with respect to the uniform measure). Since the only probabilistic part of the above algorithm is the choice of the numbers $n(e,k)$, we can assume that $\Psi_k$ splits its oracle~$X$ into blocks of bits of length $l(e,k)$, $e \in \omega$, where $2^{l(e,k)}=N(e,k)$ (for this we need $N(e,k)$ to be a power of two, which we can assume is the case without loss of generality) and interpret the block of bits of~$X$ of size $l(e,k)$ as the integer $n(e,k,X)$. 

Let $\mathcal{F}_{e,k}$ be the set of $X$'s which cause $\Psi_k$ to fail because of the $\mathcal{R}_e$-strategy having outcome (iii). We already know that $\mathcal{F}_{e,k}$ has measure at most $\frac{1}{N(e,k)}$. Analyzing the above algorithm, we see that $\mathcal{F}_{e,k}$ is the difference of two effectively open sets:  the effectively open set $\mathcal{U}_{e,k}$ of $X$'s that cause the $\mathcal{R}_e$ strategy of $\Psi_k$ to make an active guess, which is a $\Sigma_1$-event, minus the effectively open set $\mathcal{V}_{e,k}$ of $X$'s that cause the $\mathcal{R}_e$ strategy of $\Psi_k$ to make an active guess \emph{and} this active guess turns out to be correct, which is also a $\Sigma_1$ event.

If an~$X$ belongs to only finitely many $\mathcal{F}_{e,k}=\mathcal{U}_{e,k} \setminus \mathcal{V}_{e,k}$, this means that for $k$ large enough it does not belong to $\mathcal{F}_{e,k}$ for any~$e$, meaning that for sufficiently large $k$, ~$\Psi_k^X$ succeeds in producing a $1$-generic. Recall that $\mathcal{F}_{e,k}$ has measure at most $\frac{1}{N(e,k)}$, and the function~$N(\cdot,\cdot)$ can be chosen as large as we need it to be. We can take $N(e,k)=2^{-\tuple{e,k}}$ and combine the $\mathcal{F}_{e,k}$ into a single sequence of sets by taking $\mathcal{F}'_i = \mathcal{F}_{e,k}$ when $i = \tuple{e,k}$, which ensures that $\mathcal{F}'_i$ has measure at most $2^{-i}$. 

To sum up, we need~$X$ to only belong to finitely many sets in a sequence $(\mathcal{F}'_i)_{i \in \omega}$ where each $\mathcal{F}'_i$ is the difference of two effectively open sets (uniformly in~$i$) and $\mathcal{F}'_i$ has measure at most $2^{-i}$ (by the Borel-Cantelli lemma, almost every~$X$ has this property). This is very similar to the notion of difference tests, introduced by Franklin and Ng~\cite{FranklinN2011}. A difference test is precisely a sequence $(\mathcal{D}_i)_{i\in\omega}$ where each $\mathcal{D}_i$ is a difference of two effectively open sets (uniformly in~$i$) and has measure at most $2^{-i}$, just like in our case. However, the passing condition for difference randomness is weaker than what we need: $X\in\cs$ passes a difference test if it does not belong to all of the $\mathcal{D}_i$'s, and we say that~$X$ is difference random if it passes all difference tests (which, as proven by Franklin and Ng, is equivalent to $X$ being Martin-L\"of random and not computing $\emptyset'$). In our case we need $X$ to not belong to infinitely many $\mathcal{F}'_i$, the so-called \emph{Solovay passing condition}. 

This lead the authors, in early presentation of this work, to propose the notion of ``strong difference randomness", where $X$ would be said to be strongly difference random if for any difference test $(\mathcal{D}_i)_{i\in\omega}$, $X$ belong to at most finitely many $\mathcal{D}_i$'s. However, as observed by Hoyrup (private communication), this is not a robust randomness notion as it depends on the bound we put on the measure of $\mathcal{D}_i$. That is, if we defined instead a difference test by requiring that each $\mathcal{D}_i$ has measure at most $1/(i+1)^2$ instead of $2^{-i}$, the Borel-Cantelli lemma would still tell us that almost every~$X$ passes the test, but it is not clear that a test with a $1/(i+1)^2$ bound can be covered by one or several tests with a $2^{-i}$ bound. Indeed, what we would normally like to do to convert a such a test $(\mathcal{D}_i)_{i\in\omega}$ with $\mu(\mathcal{D}_i) \leq 1/(i+1)^2$ for each~$i$ into a test $(\mathcal{D}'_i)_{i\in\omega}$ with $\mu(\mathcal{D}'_i) \leq 2^{-i}$ such that an $X$ failing the first also fails the second is apply the following technique: Construct a computable sequence of integers $j_1 < j_2 < \ldots$ such that for every $i$, $\sum_{k \geq j_i} 1/(k+1)^2 \leq 2^{-i}$, which can be done since $\sum_k 1/(k+1)^2$ is a computable sum. Then for each~$i$, we set
\[
\mathcal{D}'_i = \mathcal{D}_{j_i} \cup \mathcal{D}_{j_i+1} \ldots \cup \mathcal{D}_{j_{i+1}-1}.
\]
By definition of the $j_i$'s, we have $\mu(\mathcal{D}'_i) \leq 2^{-i}$, and if an~$X$ belongs to infinitely many sets $\mathcal{D}_i$, it also belongs to infinitely many sets $\mathcal{D}'_i$. This for example would work if the $\mathcal{D}_i$'s were effectively open sets, but the problem here is that they are differences of two open sets, which may no longer be the case for the $\mathcal{D}'_i$ we constructed: we only know they are a finite unions of differences of effectively open sets (equivalently, Boolean combinations of effectively open sets as every Boolean combination can be written in this form). 

To get a more robust randomness notion that corresponds to what we need for fireworks arguments to work, we have two main options:

\begin{itemize}
\item Either we keep the bound $2^{-i}$, but allow each level $\mathcal{D}_i$ of the test to be a finite union of differences of two effectively open sets (presented as a finite list of indices for these sets uniformly in~$i$) with the Solovay passing condition. (By the above argument, in that case, the bound $2^{-i}$ can be replaced by any $a_i$ such that $\sum_{a_i}$ is finite and is a computable real number without changing the power of the family of tests).
\item Or we allow the measure bound to vary, i.e., we allow all tests of type $(\mathcal{D}_i)_{i\in\omega}$ such that each $\mathcal{D}_i$ is a difference of two effectively open sets uniformly in~$i$,   and there is a computable sequence $(a_i)_{i\in\omega}$ of rationals such that the sum $\sum_i a_i$ is finite (and possibly computable), such that $\mu(\mathcal{D}_i) < a_i$ (and again use the Solovay passing condition for these tests). 
\end{itemize}

The second approach seems to give a new randomness notion, which probably deserves to be studied. But one would first have to decide whether it is more natural to require the sum $\sum_i a_i$ to be computable or simply finite, prove that it does differ from existing randomness notions, etc. This would take us beyond the scope of this paper. 

The first approach is just as natural, and has the non-negligible advantage to take us back to an existing randomness notion: Demuth randomness. This is one of the central randomness notions between $1$-randomness and $2$-randomness, which has received a lot of attention recently; see, for example,~\cite{BienvenuDGNT2014,GreenbergT2014,KuceraN2011}. Demuth randomness is defined as follows. We fix an effective enumeration $(\mathcal{W}_e)_{e \in \omega}$ of all effectively open sets. A Demuth test is a sequence $(\mathcal{W}_{g(n)})_{n \in \omega}$ where $g$ is an $\omega$-c.a.\ function (that is, $g \leq_{wtt} \emptyset'$, or equivalently, $g$ is a $\Delta^0_2$ function which has a computable approximation $g(n,s)$ such that for every~$n$ the number of $s$ such that $g(n,s) \not= g(n,s+1)$ is bounded by $h(n)$ for some fixed computable function $h$) and $\mu(\mathcal{W}_{g(n)}) \leq 2^{-n}$ for all~$n$. $X\in\cs$ is Demuth random if it only belongs to at most finitely many levels of any given Demuth test $(\mathcal{W}_{g(n)})_{n \in \omega}$ (as shown by Ku\v cera and Nies~\cite{KuceraN2011}, the notion is independent of the bound, in that one can take any other sequence $(a_n)_{n\in\omega}$ in place of $2^{-n}$, as long as $\sum_n a_n$ is finite and computable). One can also define weak Demuth randomness by changing the passing condition: $X$ is weakly Demuth random if for any Demuth test $(\mathcal{W}_{g(n)})_{n \in \omega}$, $X \notin \bigcap_n \mathcal{W}_{g(n)}$. This is a strictly weaker notion: indeed, weak Demuth randomness is implied by weak 2-randomness (where $X\in\cs$ is weakly 2-random if it does not belong to any $\Pi^0_2$ nullset), while Demuth randomness is incomparable with it.

Our next theorem shows that the randomness notion yielded by the first approach above is indeed Demuth randomness.

%
%\begin{definition}
%A list-difference test is a test of the form $( \mathcal{L}_n)_{n \in \N}$ where $\mathcal{L}_n$ is an finite union of differences of two open sets (presented as a finite list of indices for these sets), and such that $\lambda(\mathcal{L}_n) \leq 2^{-n}$. 
%\end{definition}

\begin{theorem}\label{thm:demuth-list-diff}
$X \in \cs$ is Demuth random if and only if for every test $(\mathcal{D}_n)_{n\in\omega}$ where $\mathcal{D}_n$ is a finite union of differences of two open sets (presented as a finite list of indices for these sets), and such that $\mu(\mathcal{D}_n) \leq 2^{-n}$,~$X$ only belongs to finitely many sets $\mathcal{D}_n$. 
\end{theorem}

This is the analogue of a result of Franklin and Ng~\cite{FranklinN2014} who proved the same equivalence between this new type of tests and Demuth tests when the passing condition is ``$X$ does not belong to all levels", thus obtaining a new characterization of weak Demuth randomness. 

\begin{proof}
We first see how to turn a Demuth $(\mathcal{W}_{g(n)})_{n \in \omega}$ into a test $(\mathcal{D}_n)_{n \in \N}$ as in the theorem. Let~$h$ be computable bound on the number of changes of~$g$. For each~$n$, effectively create a list of $h(n)$ pairs of difference sets $\mathcal{U}_k \setminus\mathcal{V}_k$, where $\mathcal{U}_k$ is the $k$-th version of $\mathcal{W}_{g(n)}$ (that is, $\mathcal{U}_k$ is equal to $\mathcal{W}_{g(n,s+1)}$, where $s+1$ is the $k$-th stage at which $g(n,s+1)\not=g(n,s)$; if there is no such~$s$, $\mathcal{U}_k$ remains empty), and $\mathcal{V}_k$ is equal to $\mathcal{U}_k$ if a $(k+1)$-th version of $\mathcal{W}_{g(n)}$ ever appears, and empty otherwise. It is easy to see that all $\mathcal{U}_k \setminus\mathcal{V}_k$ are empty, except the one where $k$ corresponds to the final version of $\mathcal{W}_{g(n)}$, for which  we have $\mathcal{U}_k \setminus \mathcal{V}_k=\mathcal{W}_{g(n)}$. Calling $\mathcal{D}_n$ the finite union of the $\mathcal{U}_k \setminus \mathcal{V}_k$, we have $\mathcal{D}_n = \mathcal{W}_{g(n)}$. Thus~$X$ is in infinitely many $\mathcal{W}_{g(n)}$ if and only if it is in infinitely many $\mathcal{D}_n$. 

Conversely, let us see how to convert a test  $(\mathcal{D}_n)_{n \in \omega}$ as above into a Demuth test. For every $n$, the $n$-th level $\mathcal{W}_{g(n)}$ of the Demuth test is built as follows: consider the list of $\mathcal{U}_k \setminus  \mathcal{V}_k$ composing $\mathcal{D}_{n+1}$. There are $h(n)$ such difference sets, where $h$ is a computable function. The first version of $\mathcal{W}_{g(n)}$ is $\bigcup_{k} \mathcal{U}_k$. Meanwhile, enumerate all $\mathcal{V}_k$ in parallel. When we see at some stage~$s$ that the measure of one of the $\mathcal{V}_k$'s becomes greater than some new multiple of $2^{-n-1}/h(n)$, we change the version of $\mathcal{W}_{g(n)}$: The new version is now $\bigcup_{k} (\mathcal{U}_k \setminus \mathcal{V}_k[s])$. 

The number of versions of $\mathcal{W}_{g(n)}$ is bounded by $h(n)^2 \cdot 2^{n+1}$. Indeed, each $\mathcal{V}_k$ can reach a new multiple of $2^{-n-1}/h(n)$ only $h(n) \cdot 2^{n+1}$ times (causing a new version of $\mathcal{W}_{g(n)}$), and there are $h(n)$ sets $\mathcal{V}_k$'s. By definition, every version of $\mathcal{W}_{g(n)}$ contains $\mathcal{D}_{n+1}$, so any sequence contained in infinitely many $\mathcal{D}_{n}$'s is contained in infinitely many $\mathcal{W}_{g(n)}$. It remains to evaluate the measure of $\mathcal{W}_{g(n)}$. Let $s$ be the stage at which the final version of $\mathcal{W}_{g(n)}$ has appeared. Since this is the final version, this means that no $\mathcal{V}_k$ will increase by more than $2^{-n-1}/h(n)$ in measure after stage~$s$, thus for each $k$,
\[
\Big|\mu(\mathcal{U}_k \setminus \mathcal{V}_k[s]) - \mu(\mathcal{U}_k \setminus \mathcal{V}_k)\Big| \leq 2^{-n-1}/h(n).
\] 
Since there are $h(n)$ terms $\mathcal{U}_k \setminus \mathcal{V}_k$ in $\mathcal{D}_n$, we have
\[
\Big|\mu(\bigcup_k \mathcal{U}_k \setminus \mathcal{V}_k[s]) - \mu(\bigcup_k \mathcal{U}_k \setminus \mathcal{V}_k)\Big| \leq 2^{-n-1},
\]
i.e., 
\[
\Big|\mu(\mathcal{W}_{g(n)}) - \mu(\mathcal{D}_{n+1})\Big| \leq 2^{-n-1}.
\] 
As $\mu(\mathcal{D}_{n+1}) \leq 2^{-n-1}$, we get $\mu(\mathcal{W}_{g(n)}) \leq 2^{-n-1}+2^{-n-1}=2^{-n}$ as desired. 
\end{proof}

From this theorem and our analysis of the tests induced by fireworks arguments, we immediately get that Demuth randomness is sufficient to compute a $1$-generic. 

\begin{theorem}\label{thm:demuth-1-generic}
Every Demuth random computes a $1$-generic. 
\end{theorem}

Theorem~\ref{thm:demuth-list-diff} more generally tells us that if $\mathcal{S} \subseteq \cs$ is such that a member of $\mathcal{S}$ can be obtained with positive probability via a fireworks argument, one can conclude that every Demuth random~$X$ computes some element of~$\mathcal{S}$. For example, more intricate fireworks arguments were used in~\cite{BienvenuP2017} for the set $\mathcal{S}$ consisting of the diagonally non-computable (DNC) functions which compute no Martin-L\"of random. Together with the present paper, we have established that every Demuth random~$X$ computes a DNC function which computes no Martin-L\"of random.

This theorem cannot be generalized to currently available definitions of randomness (other than those that imply Demuth randomness): any notion implied by weak 2-randomness, which includes weak Demuth randomness, difference randomness, Martin-L\"of randomness, Oberwolfach randomness, balanced randomness, etc., do not guarantee the computation of a 1-generic. Indeed, every $1$-generic is hyperimmune while there are sequences that are both weakly 2-random and of hyperimmune-free Turing degree~\cite[Proposition 3.6.4]{Nies2009}. 

One of the main results of~\cite{BarmpaliasDL2013} is that every non-computable $X$ which is merely \emph{Turing below} a $2$-random~$Y$ computes a $1$-generic (Kurtz~\cite{Kurtz1981} had proven this result for almost all~$Y$, but the exact level of randomness needed was unknown). One cannot replace $2$-randomness by Demuth randomness in this statement. Indeed, take a~$Y$ which is Demuth random but not weakly 2-random. By a result of Hirschfeldt and Miller, $Y$ computes a non-computable c.e.\ set~$A$ (see~\cite[Corollary 7.2.12]{DowneyH2010}). Applying a second result, due to Yates, that every non-computable c.e.\ set computes some~$X$ of minimal degree, it follows that a sequence~$X$ that is below a Demuth random does not compute any $1$-generic, as no $1$-generic has minimal degree. 

\subsection{Demuth randomness vs stronger genericity notions}

So far we have seen that every Demuth random computes a $1$-generic sequence, and that this was in some sense optimal among the randomness notions that have been considered in the litterature. One may wonder whether the same is true of the genericity notion; that is, can we improve Theorem~\ref{thm:demuth-1-generic} by replacing 1-genericity with a stronger genericity notion? Again, we give a negative answer for the genericity notions we are aware of between $1$-genericity and $2$-genericity.  We have seen in the introduction that a good candidate for a possible strengthening of Theorem~\ref{thm:demuth-1-generic} would be weak 2-genericity. We show that at this level of genericity, the situation changes significantly: a Demuth random sequence and a weakly 2-generic sequence always form a minimal pair. In fact, we will show this even for a weaker notion, known as pb-genericity, introduced by Downey, Jockusch and Stob~\cite{DowneyJS1996}.

\begin{definition}
$G \in \cs$ is pb-generic if for every function~$f: \str \rightarrow \str$ such that 
\begin{itemize}
\item[(i)] $f$ is computable in~$\emptyset'$ with a primitive recursive bound on the use, and
\item[(ii)] $\sigma \prefx f(\sigma)$ for all~$\sigma$,
\end{itemize}
there are infinitely many~$n$ such that $f(G \uh n) \prefx G$. 
\end{definition}

It is easy to see that weak 2-genericity implies pb-genericity. Indeed, for each $n$, consider the set of strings $S_n=\{f(\sigma) \mid |\sigma| \geq n\}$. By definition of~$f$, $S_n$ is both dense and $\emptyset'$-c.e., thus a weakly 2-generic sequence~$G$ must have a prefix in every $S_n$, which is exactly what it means for $G$ to be pb-generic. 

It is already known that a Demuth random cannot compute a pb-generic: indeed, Downey et al.~\cite{DowneyJS1996} proved that  $X \in \cs$ computes a pb-generic if and only if it has array non-computable degree, meaning that $X$ can compute a total function $f: \N \rightarrow \N$ which is dominated by no $\omega$-c.a.\ function $g$ (where $g$ is said to dominate $f$ if $f(n) \leq g(n)$ for almost all~$n$), and Downey and Ng~\cite{DowneyN2009} showed that all Demuth randoms have array computable degree. The next theorem improves this. 

\begin{theorem}
If $X$ is Demuth random and~$G$ is pb-generic, then $X$ and $G$ form a minimal pair in the Turing degrees. 
\end{theorem}

This also strengthens the theorem of Nies et al.'s mentioned in the introduction which asserts that for any pair $(X,G)$ consisting  of a $2$-random and a $2$-generic forms a minimal pair. As we shall see later, among the available effective notions of randomness and genericity in the literature (that we are aware of), this theorem is the best we can get. 

\begin{proof}
For this proof, we fix a primitive recursive bijection $\Nat$ from strings to integers. Let us consider a pair $(\Phi, \Psi)$ of Turing functionals.  We want to show that if $\Phi^G$ and $\Psi^X$ are both defined and equal, then they are computable. For each such pair of functionals, we exploit the pb-genericity of~$G$ by building  a specific function~$f$, and exploit the Demuth randomness of~$X$, by building a Demuth test $(\W_{g(n)})_{n \in \N}$. The function~$f$ is defined as follows. For a given~$\sigma\in\str$, look for a family of pairs of strings $(\sigma_1,\tau_1), ..., (\sigma_{2^N},\tau_{2^N})$, where $N=\Nat(\sigma)$, such that:
\begin{itemize}
\item all $\sigma_i$ strictly extend $\sigma$,
\item for all~$i$, $\Phi^{\sigma_i} \exts \tau_i$,
\item and the $\tau_i$ are pairwise incomparable. 
\end{itemize}

Note that these conditions are c.e.\ and therefore if there is such a family, it can be effectively found. If such a family is eventually found, let $f(\sigma)=\sigma_j$ where $j$ is the smallest index such that the measure of $\Psi^{-1}(\tau_j)$ (that is, the effectively open set $\{X \in \cs \mid \Psi^X \exts \tau_j\}$) is smaller or equal to $2^{-N}$ (note that since the $\tau_i$ are mutually incomparable, the open sets $\Psi^{-1}([\tau_i])$ are disjoint, hence at least one of them must have measure at most $2^{-N}$) and define $\W_{g(N)}$ to be equal to $\Psi^{-1}([\tau_j])$. If there is no such family of pairs of strings, set $f(\sigma)=\sigma$ 	and $\W_{g(N)}$ to be the empty set. We show that this construction works via a series of claims. \\

\begin{claim}
The function~$f$ is computable in $\emptyset'$ with a primitive recursive bound on the use. 
\end{claim}

\begin{proof}
Indeed $f$ has a computable approximation with a primitive recursive bound on the number of mind changes. For a given $\sigma$ and stage~$s$, let $f(\sigma)[s]$ be equal to $\sigma$ if no family of pairs as in the construction has been found by stage~$s$ and in case such a family was found before stage~$s$, set $f(\sigma)[s]$ be equal to $\sigma_k$ where $k$ is the minimal such that the measure of $\Psi^{-1}(\tau_k)[s]$ is smaller or equal to $2^{-N}$ (where, again, $N=\Nat(\sigma)$). Note that this is indeed a computable approximation, and $f(\sigma)[t]$ can change at most $2^{N}$ times. Indeed, it changes once when (and if) the family of pairs is found, and changes every time the current candidate $\tau_k$ is discovered to be such that the measure of $\Psi^{-1}(\tau_k)[s]$ is bigger than $2^{-N}$, which can only happen to $2^{N}-1$ strings. Moreover, $\Nat$ is a primitive recursive function, thus so is $\sigma \mapsto 2^{\Nat(\sigma)}$. 
\end{proof}

\begin{claim}
The sequence $(\mathcal{W}_{g(N)})_{N \in \N}$ is a Demuth test. 
\end{claim}

\begin{proof}
For a given~$N$, let $\sigma$ be the string such that $N=\Nat(\sigma)$. By definition of $\W_{g(N)}$, we have $\mu(\W_{g(N)}) \leq 2^{-N}$. Moreover, $g$ has a computable approximation with at most $2^N$-many changes, the proof of this being almost identical to that of the previous claim. Initially, and at any stage~$s$ before the desired family of strings is found, $\W_{g(N)[s]}$ is empty, and at any stage~$s$ posterior to finding such a family, one can take $\W_{g(N)[s]}=\Psi^{-1}(\tau_k)[s]$, where~$k$ is minimal such that $\mu(\Psi^{-1}(\tau_k)[s])\leq2^{-N}$. It follows that $g(N)$ can change at most $2^{N}$ many times. 
\end{proof}

\begin{claim}
If there are infinitely many~$n$ such that $f(G \uh n) \prefx G$ and if $X$ passes the Demuth test $(\mathcal{W}_{g(N)})_{N \in \N}$, then either $\Phi^G$ is partial, or $\Phi^G$ is computable, or $\Phi^G \not= \Psi^X$. 
\end{claim}

\begin{proof}
By definition of `passing a Demuth test', we know that $X$ only belongs to finitely many~$\mathcal{W}_{g(n)}$. Therefore let us take $n$ such that $f(G \uh n) \prefx G$ and large enough so that $X \notin \mathcal{W}_{g(N)}$, where $N=\Nat(G \uh n)$. Let $\sigma=G \uh n$. We distinguish two cases.\\

\noindent Case 1: $f(\sigma)=\sigma$. By definition of~$f$, this means that no family of pairs was ever found for that~$\sigma$, meaning that the set of strings
\[
T= \{\tau \mid (\exists \sigma' \exts \sigma)\, \Phi^{\sigma'} \exts \tau \}
\] 
contains at most $2^{N}-1$ incomparable strings. Observe that~$T$ is a c.e.\ tree (the computable enumerability is obvious by definition, and it is clearly closed under the prefix relation). Therefore, all infinite paths of~$T$ are strongly isolated, in the sense that for every infinite path~$X$, there is an $n_0$ such that for all ~$n \geq n_0$, $X \uh (n+1)$ is the only extension of~$X \uh n$ in~$T$. Indeed, if this were not the case for some path~$X$, we would have infinitely many $n$ such that $\tau_n = (X \uh n) ^\frown (1-X(n))$ is in the tree. Since for any $n \neq m$ $\tau_n$ and $\tau_m$ are incomparable, this would contradict our assumption that there is no family of $2^N$ incomparable strings in the tree. Of course, any strongly isolated path in a c.e.\ tree is computable since for almost all~$n$, $X(n)$ can be effectively found from $X \uh n$. Now observe that since $\sigma$ is a prefix of~$G$, $\Phi^G$, if it is defined, is a path of~$T$. Thus $\Phi^G$ is either undefined or is a computable sequence. \\

\noindent Case 2: $f(\sigma)$ strictly extends $\sigma$. By construction, this means that there is a string~$\tau$ such that $\Phi^G \exts \Phi^{f(\sigma)} \exts \tau$ and $\W_{g(N)} = \Psi^{-1}(\tau)$. Since by assumption $X \notin \W_{g(N)}$, this means that $\Psi^X \nsucceq \tau$, and thus $\Phi^G \not= \Psi^X$. 
\end{proof}

This last claim completes the proof. 
\end{proof}

\section{Weak 2-randomness vs genericity}

The other main randomness notion below $2$-randomness, namely weak 2-randomness, behaves quite differently from Demuth randomness in terms of the ``escaping power" of the functions $f:\omega\rightarrow\omega$ that such random elements can compute. Following the terminology of~\cite{AndrewsGM2014}, given $\mathscr{F}$ a countable family of functions, we say that a function~$g$ is  $\mathscr{F}$-escaping if it is not dominated by any function $f \in \mathscr{F}$. We will also say that $X \in \cs$ has $\mathscr{F}$-escaping degree if it computes an $\mathscr{F}$-escaping function. For example, $X$ has $\Delta^0_1$-escaping degree iff it has hyperimmune degree and $X$ has $(\omega$-c.a.)-escaping degree iff it has array non-computable degree. 

For Demuth random sequences we have an upper and a lower bound on escaping power: every Demuth random sequence has $\Delta^0_1$-escaping degree (see~\cite{Nies2009}; this also follows from the fact that every Demuth random computes a $1$-generic), but no Demuth random is $(\omega$-c.a.)-escaping as mentioned in the previous section (a fortiori, no Demuth random is $\Delta^0_2$-escaping). 

By contrast, weak 2-randomess is completely orthogonal to this measure of computational strength. On the one hand, some weakly 2-random sequences have hyperimmune-free degree (see for example~\cite{Nies2009}). On the other hand, a striking result by Barmpalias, Downey and Ng~\cite{BarmpaliasDN2011} is that for \emph{any} countable family $\mathscr{F}$ of functions, there is a weakly 2-random~$X$ that has has $\mathscr{F}$-escaping degree. 

There are close connections between escaping degrees and the ability to compute generics: 
\begin{itemize}
\item $X$ computes a weak-1-generic iff $X$ has $\Delta^0_1$-escaping degree~\cite{Kurtz1981}, 
\item $X$ computes a pb-generic iff it has $(\omega$-c.a.)-escaping degree~\cite{DowneyJS1996},
\item $X$ computes a weak-2-generic iff it has $\Delta^0_2$-escaping degree~\cite{AndrewsGM2014}, 
\item If $X$ has $\Delta^0_3$-escaping degree, it computes a $2$-generic~\cite{AndrewsGM2014}.
\end{itemize}

Together with the theorem of Barmpalias, Downey and Ng, the last item shows that there exists a weak-2-random which computes a $2$-generic. However, a very interesting result from~\cite{AndrewsGM2014} is that we cannot extend this correspondence much further in the genericity hierarchy: indeed, for \emph{any} countable family  $\mathscr{F}$ of functions, there exists an~$\mathscr{F}$-escaping function~$g$ which computes no weakly 3-generic sequence. Thus the theorem of Barmpalias, Downey and Ng does not say how weak 2-randomness interacts with weak 3-genericity or higher genericity notions. Our next theorem strengthens their result to show that in fact, there always exists a weakly 2-random sequence that computes a generic sequence, no matter how strong the genericity notion is.

\begin{theorem}\label{thm:w2r-may-compute-generic}
Let $\mathcal{G}$ be a comeager subset of $\cs$. There exists a weakly 2-random~$X$ that computes some $Y \in \mathcal{G}$. 
\end{theorem}

The rest of this section is dedicated to the proof of Theorem~\ref{thm:w2r-may-compute-generic}. The main ideas are the same as the ones use by Barmpalias et al.\ in~\cite{BarmpaliasDN2011}, and there is little doubt that they would have been able to refine their proof to get Theorem~\ref{thm:w2r-may-compute-generic} had they been considering the problem of coding generics into randoms. Nonetheless, some adaptations are needed, and this is what we provide below. Also, this is more an expository choice, but our proof differs from Barmpalias et al.'s by the characterization we use of weak 2-randomness: while they used the fact that an $X \in \cs$ is weakly 2-random iff it is Martin-L\"of random and forms a minimal pair with $\emptyset'$ (see~\cite{DowneyH2010}), we directly use the definition of weak 2-randomness, that is, $X$ is weakly 2-random iff it does not belong to any $\Pi^0_2$ nullset. 

The main tool we need for our proof is the so-called Ku\v cera-G\'acs coding, which allows one to encode any information into a Martin-L\"of random real. Let us review the basic mechanisms of this technique. 

Ku\v cera-G\'acs coding begins by fixing a $\Pi^0_1$ class $\mathcal{R}$ containing only Martin-L\"of random sequences (in particular, $\mathcal{R}$ has positive measure). Ku\v cera proved that this class has the following property: There exists a computable function~$h$ such that for any $\Pi^0_1$ class~$\mathcal{P}$ contained in $\mathcal{R}$, and any $\sigma \in \str$,
\[
[\sigma] \cap \mathcal{P} \not= \emptyset\  \Leftrightarrow \ \mu([\sigma] \cap \mathcal{P}) > 2^{-h(\sigma,\mathcal{P})}
\]
(in the above equivalence and in what follows, a $\Pi^0_1$ class as an argument should be read as an index for this $\Pi^0_1$ class). In particular this means that when $[\sigma] \cap \mathcal{P} \not= \emptyset$, $\sigma$ has at least two extensions~$\tau$ of length $h(\sigma,\mathcal{P})$ such that $[\tau] \cap \mathcal{P} \not= \emptyset$. The Ku\v cera-G\'acs coding, in its simpler form, consists in coding $0$ by $\tau_{\mathit{left}}$, the leftmost such $\tau$ and $1$ by the $\tau_{\mathit{right}}$ rightmost one. Indeed, knowing $\sigma$, $\P$, and given a $\tau$ which is either the leftmost or rightmost string of length $h(\sigma,\mathcal{P})$, we can figure out which is which because the strings $\tau$ such that $[\tau] \cap \mathcal{P} \not= \emptyset$ form a co-c.e.\ set. One can then encode a second bit by an extension of $\tau$ of length $h(\tau, \P)$, and iterate the process above $\tau$. If we were to continue this process indefinitely, since coding is monotonic (each time we encode one more bit the new code word is an extension of the previous one), we can take the union of all the codewords to get a sequence~$X \in \cs$ from which we can can computably recover all the bits we encoded during the construction. Since at  each step of the process we ensure that the new codeword $\tau$ satisfies $[\tau] \cap \P \not= \emptyset$, this means that $X$ has arbitrarily long prefixes $X \uh n$ such that $[X \uh n] \cap \P \not= \emptyset$, and thus $X \in \P$, as $\P$ is a closed set. 

Coming back to finite encoding, the Ku\v cera-G\'acs technique gives us a (non-computable) function $KG: \str \times \str \times \omega \rightarrow \str$ such that $KG(\xi \mid \sigma, \P)$ is the encoding of the string~$\xi$ above $\sigma$ within the $\Pi^0_1$ class $\P$ following the above technique, and thus enjoying the following properties for all $\xi$, $\sigma$, and $\P$ a $\Pi^0_1$ subset of $\mathcal{R}$:
\begin{itemize}
\item $KG(\xi \mid \sigma,\P) \exts \sigma$
\item $[\sigma] \cap \P \not= \emptyset$, then $[KG(\xi \mid \sigma, \P)]  \cap \P \not= \emptyset$
\item $KG(\;\cdot \mid \sigma, \P)$ is one-to-one for every fixed $\sigma, \P$; furthermore, up to composing with a prefix-free encoding of $\str$, we can assume that for a fixed $(\sigma, \P)$, the range of $KG(\;\cdot \mid \sigma, \P)$ is prefix-free.
\item There exists an effective `decoding' procedure, which we denote by $KG^{-1}$, which is a partial computable function such that (a) $KG^{-1}(\tau \mid \sigma,\P)=\xi$ when $\tau=KG(\xi \mid \sigma,\P)$ and (b) for a fixed $(\sigma, \P)$, the domain of $KG^{-1}(\;\cdot \mid \sigma, \P)$ is prefix-free.
%\item there exists a decoding function, i.e., a partial computable function $D(.|.,.)$ such that for any $(\sigma, \P)$ with $[\sigma] \cap \P \not= \emptyset$, $D\big(KG(\xi \mid \sigma, \P)\big|\sigma,\P\big)=\xi$
\end{itemize}

Now, we want to encode information into a weakly 2-random sequence. Of course since a weakly 2-random sequence cannot compute any non-computable $\Delta^0_2$ set, we cannot hope for an encoding which can be perfectly decoded and thus the decoding procedure will be allowed to make errors.

 The idea is to sequentially use Ku\v cera-G\'acs codings where the $\Pi^0_1$ class shrinks at each step in order to make the union of the codewords a weakly 2-random sequence.   To do this, let $(\U^e_k)_{e,k\in\omega}$ be an effective enumeration of $\Sigma^0_1$ subsets of $\cs$ such that every $\Pi^0_2$ set is equal to $\bigcap_k \U^e_k$ for some $e$. Without loss of generality, we can ensure that $\U^e_{k+1} \subseteq \U^e_k$ for all~$e,k$. We let $\C^e_k$ be the complement of $\U^e_k$. We also let $e^*_1< e^*_2 < \ldots$ the sequence of indices~$e$ such that $\bigcap_k \U^e_k$ is a nullset. This is, of course, not a computable sequence, but the idea is to make this sequence part of the encoded information. 

Let $g: \omega \times \str \times \omega \rightarrow \omega \cup \{\infty\}$ be the function defined by 
\[
g(e, \sigma, \P) = \inf \bigl\{k \mid  [\sigma] \cap \C^e_k \cap \P \not= \emptyset\bigr\}.
\]
Observe that $g$ is lower semi-computable. 

Let us fix a computable, one-to-one pairing function $\tuple{ \cdot,\cdot}:\omega\times\str\rightarrow\str$.  Given a sequence of strings $\xi_1, \ldots , \xi_k$, its \emph{W2R-encoding}, denoted by $E(\xi_1, \ldots, \xi_k)$ is the string $\tau_1\tau_2\cdots \tau_k$ where
\[
\left\{
\begin{array}{l}
\P_0= \mathcal{R}\\
\tau_1 = KG(\tuple{e^*_1, \xi_1} \mid \Lambda, \P_0)
\end{array}
\right.
\]
and for $1 \leq n < k$,
\[
\left\{
\begin{array}{l}
\P_{n} = \P_{n-1} \cap \C^{e^*_n}_{g(e^*_n,\tau_n,\P_{n-1})}\\
\tau_{n+1} = KG(\tuple{e^*_{n+1}, \xi_{n+1}} \mid \tau_n, \P_{n}).
\end{array}
\right.
\]

By construction, $E(\xi_1, \ldots, \xi_{k})$ is a prefix of $E(\xi_1, \ldots, \xi_{k+1})$ for all~$k$, so we can extend the definition of $E$ to infinite sequences of strings $(\xi_i)_{i \in \omega}$ by setting 
\[
E\bigl((\xi_i)_{i \in \omega}\bigr) = \bigcup_k E(\xi_1, \ldots, \xi_k).
\]

Moreover, the construction ensures that  $E\bigl((\xi_i)_{i \in \omega}\bigr)$ belongs to all $\P_n$, and $\P_{n+1}$ is chosen to be disjoint from $\bigcap_k \U^{e^*_n}_k$, the $n$-th $\Pi^0_2$ nullset.  Thus, $E\bigl((\xi_i)_{i \in \omega}\bigr)$ is weakly 2-random for any sequence $(\xi_i)_{i \in \omega}$. 

Let us now define a `decoding' functional $\Gamma$. This functional will make `errors' in the decoding process, i.e., we will not have $\Gamma^{E((\xi_i)_{i \in \omega})}=\xi_1\xi_2\cdots$. However, we will ensure the following property: if $\xi_1, \ldots, \xi_k$ are fixed, there is an $r$ such that for any  extension $\xi_{k+1}\xi_{k+2}\cdots$ of the sequence, the prefix of $\Gamma^{E((\xi_i)_{i \in \omega})}$ of size $|\xi_1\xi_2\cdots\xi_{k+1}|$ differs from $\xi_1\xi_2\cdots\xi_{k+1}$ on at most~$r$ bits. 

The procedure $\Gamma$ is defined as follows. On input~$X$, for all~$t$ in parallel, $\Gamma$ runs a sub-procedure (which we call a \emph{$t$-sub-procedure}) that tries to find a prefix $\tau_1 \cdots \tau_k$ of~$X$ and a sequence of triples $\tuple{a_n,\zeta_n,\Q_n}_{1 \leq n \leq k}$  such that
\[
\left\{
\begin{array}{l}
\Q_0= \mathcal{R}\\
\tuple{a_1, \zeta_1} = KG^{-1}(\tau_1 | \Lambda, \Q_0)
\end{array}
\right.
\]
and for $1 \leq n < k$,
\[
\left\{
\begin{array}{l}
\Q_{n} = \Q_{n-1} \cap \C^{a_n}_{g(a_n,\tau_n,\Q_{n-1})[t]}\\
\tuple{a_{n+1}, \zeta_{n+1}} = KG^{-1}(\tau_{n+1} | \tau_n, \Q_{n}).
\end{array}
\right.
\]
Note that there is at most one such sequence because $KG$ is prefix-free and one-to-one for each fixed pair of conditions $(\sigma,\P)$. If such a sequence is found, then setting $\zeta=\zeta_1\cdots\zeta_k$,   $\Gamma^X(i)$ is defined to be $\zeta(i)$ for any~$i \leq \min(t,|\zeta|-1)$ on which $\Gamma^X(i)$ has not yet been defined by other sub-procedures with parameter $t'< t$. 

We now prove two claims which are going to allow us to conclude the proof.

\begin{claim}
For any sequence of strings $\xi_1, \ldots, \xi_k$, there exists a $N\in\omega$ such that for any $\xi_{k+1}$, if $\xi=\xi_1 \cdots \xi_{k+1}$ has length at least~$N$, then $\Gamma^{X}(i)=\xi(i)$ for any $X$ extending $E(\xi_1,\dots,\xi_{k+1})$ and $i \geq N$.  
\end{claim}

\begin{proof}
Fix $\xi_1, \ldots, \xi_k$, let $\xi_{k+1}$ be any string, and let $X$ be an infinite sequence extending $E(\xi_1, \ldots, \xi_{k+1})$. 
Let $\tau_1, \ldots, \tau_{k+1}$ and $\P_1, \ldots \P_{k+1}$ be the strings and $\Pi^0_1$ classes inductively built in the definition of $E(\xi_1, \ldots, \xi_{k+1})$. Recall that the function~$g$ is lower semi-computable and $g(e^*_n,\tau_n,\P_{n-1})$ is always finite, thus there is an~$N$ such that
\[
g(e^*_n,\tau_n,\P_{n-1})[t]=g(e^*_n,\tau_n,\P_{n-1})
\]
 for all $n \leq k+1$ and all $t \geq N$.  

This means that for $t \geq N$, the $t$-sub-procedure of $\Gamma$ will eventually find the desired sequence $\tuple{a_n,\zeta_n,\Q_n}_{1 \leq n \leq k+1}$  because $\tuple{e^*_n,\xi_n,\P_n}_{1 \leq n \leq k+1}$ satisfies that property. By uniqueness, we must have $\zeta_n=\xi_n$, $a_n=e^*_n$ and $\Q_n=\P_n$ for $n \leq k+1$. By definition of $\Gamma$, the $t$-subprocedure for $t < N$ can only define $\Gamma^X(i)$ for $i < N$. Thus if $i \geq N$ one must have $\Gamma^X(i)=\xi(i)$ where $\xi=\xi_1\cdots\xi_{k+1}$. This proves our claim. 
\end{proof}

\begin{claim}
Let $\xi_1, \ldots, \xi_k$ be fixed and let $\U$ be a dense open set. There exists $\xi_{k+1}$ such that $\Gamma^{E(\xi_1, \ldots, \xi_{k+1})} \exts \sigma$ for some $\sigma$ such that $[\sigma] \subseteq \U$. 
\end{claim}

\begin{proof}
By the previous claim, there exists an~$N\in\omega$ such that for any $\xi_{k+1}$, $\Gamma^X(i)=\xi(i)$ for all $i \geq N$, where $\xi=\xi_1\cdots\xi_{k+1}$. We can assume that $N \geq |\xi_1\cdots\xi_{k}|$.

Now, given a string $\eta\in\str$, we denote by $\U_{\eta}$ the set $\{Z \mid \eta Z \in \U\}$. Since $\U$ is dense, it is in particular dense above $\eta$, so $\U_\eta$ is dense. Consider the open set $\V = \bigcap_{|\eta|=N} \U_\eta$. A finite intersection of dense open sets is dense and, in particular, non-empty, so there must a $\zeta$ such that $[\zeta] \subseteq \bigcap_{|\eta|=N} \U_\eta$, which is equivalent to saying that $[\eta\zeta] \subseteq \U$ for all $\eta$ of length~$N$. Thus, it suffices to choose $\xi_{k+1}$ so that the bits $\xi=\xi_1\cdots\xi_{k+1}$ after position~$N$ are an extension of~$\zeta$ to get the desired result. 

\end{proof}

This last claim is just what we need to complete the proof of Theorem~\ref{thm:w2r-may-compute-generic}. Let $\mathcal{G}$ be comeager and $(\U_k)_{k \in \N}$ a family of dense open sets such that $\bigcap_k \U_k \subseteq \mathcal{G}$. The previous claim allows us to construct by induction a sequence $(\xi_k)_{k \in \N}$ of strings such that for all~$k$, $\Gamma^{X}$ is guaranteed to be in~$\U_k$ when $X$ extends $E(\xi_1,\ldots,\xi_k)$ and $\Gamma^X$ is total. Thus, taking $X=E((\xi_k)_{k \in \N})$, we have that $\Gamma^X$ is total, belongs to all $\U_k$, and as explained earlier on, $X$ must be weakly 2-random. Our theorem is proven.

\section{Conclusion}

The following table recaps the various interactions between randomness and genericity discussed in the paper: \\

\begin{tabular}{|c|c|c|c|c|}
\hline
 & $n$-gen. ($n \geq 2$) & weakly 2-gen.\ & pb-gen.\ & 1-gen.\ \\
 \hline
$n$-random ($n \geq 2$) & min.\ pair & min.\ pair & min.\ pair & computes\\
\hline
weakly 2-random & may compute & may compute  & may compute & may compute\\
\hline 
Demuth random & min.\ pair & min.\ pair & min.\ pair & computes\\
\hline
$1$-random & may compute & may compute & may compute & may compute\\
\hline
\end{tabular}
$ $ \\

\bigskip

\noindent For a given pair consisting of a randomness notion and a genericity notion:
\begin{itemize}
\item `min.\ pair' means that for any random $X$ and any generic $G$, $(X,G)$ forms a minimal pair in the Turing degrees;
\item `may compute' means that there is a random~$X$ and a generic~$G$ such that $X$ computes~$G$; and
\item `computes' means that any random $X$ computes some generic~$G$.
\end{itemize}

\vspace{5mm}

We note that these three cases do not form an exhaustive list of possibilities. It would for example be interesting to find natural pair of one randomness notion and one genericity notion such that a random never computes a generic but that a random and a generic do not necessarily form a minimal pair.

%\laurent{Question: if $A$ forms a minimal pair with $\emptyset$', is there a weak-2-random $X$ that computes~$A$?}

\bibliographystyle{alpha}
\bibliography{demuthgenericity}

\end{document}